\documentclass[reqno,b5paper]{amsart}
\usepackage{amsmath}
\usepackage{amssymb}
\usepackage{amsthm}
\usepackage{enumerate}
\usepackage[mathscr]{eucal}
\usepackage{amsfonts}
\usepackage{hyperref}
\usepackage{color, url}

\newtheorem{theorem}{Theorem} 
 
\newtheorem{Cor}{Corollary} 
\newtheorem{Lem}{Lemma}

\def\sts#1#2{\left\{#1\atop#2\right\}}

\begin{document}
\title[Truncated Euler polynomials]{Truncated Euler polynomials}
\author[Takao Komatsu \and Claudio de J. Pita Ruiz V.]{Takao Komatsu* \and Claudio de J. Pita Ruiz V.**} 
\newcommand{\acr}{\newline\indent}
\address{\llap{*\,}School of Mathematics and Statistics\acr
Wuhan University\acr
Wuhan 430072\acr
CHINA}
\email{komatsu@whu.edu.cn}
\address{\llap{**\,}Universidad Panamericana\acr
Mexico City\acr
MEXICO}
\email{cpita@up.edu.mx}
\thanks{The first author was supported in part by the grant of Wuhan University and by the grant of Hubei Provincial Experts Program.} 
\begin{abstract}
We define a truncated Euler polynomial $E_{m,n}(x)$ as a generalization of
the classical Euler polynomial $E_n(x)$. In this paper we give its some
properties and relations with the hypergeometric Bernoulli polynomial.
\end{abstract}
\maketitle
\section{Introduction}

For non-negative integer $m$, define \textit{truncated Euler polynomials} $E_{m,n}(x)$ by 
\begin{equation}
\frac{\frac{2 t^m}{m!}e^{x t}}{e^t+1-\sum_{j=0}^{m-1}\frac{t^j}{j!}}
=\sum_{n=0}^\infty E_{m,n}(x)\frac{t^n}{n!}\,.  
\label{def:tep}
\end{equation}
When $m=0$, $E_{n}(x)=E_{0,n}(x)$ is one of the definitions of Euler
polynomials, given by 
\begin{equation}
\frac{2 e^{x t}}{e^t+1}=\sum_{n=0}^\infty E_{0,n}(x)\frac{t^n}{n!}\,.
\label{def:eulerpoly}
\end{equation} 
When $x=0$ in (\ref{def:tep}), $E_{m,n}=E_{m,n}(0)$ are called {\text truncated Euler numbers}, given by  
$$ 
\frac{\frac{2 t^m}{m!}}{e^t+1-\sum_{j=0}^{m-1}\frac{t^j}{j!}}
=\sum_{n=0}^\infty E_{m,n}\frac{t^n}{n!}\,. 
$$ 

Incomplete Bernoulli numbers \cite{KLM} and incomplete Cauchy numbers \cite{Ko5,KMS} are similar truncated numbers. Both of them are based upon the restricted and associated Stirling numbers of the second kind, and the restricted and associated Stirling numbers of the first kind. The restricted Bernoulli numbers $\mathfrak{B}_{n,\le m}$ and the associated Bernoulli numbers $\mathfrak{B}_{n,\ge m}$ can be defined by 
\begin{equation*}
\frac{\log\mathfrak{E}_m(-t)}{\mathfrak{E}_m(-t)-1}=\sum_{n=0}^\infty\mathfrak{B}_{n,\le m}\frac{t^n}{n!}
\end{equation*}
and 
\begin{equation*} 
\frac{\log\bigl(1+e^{-t}-\mathfrak{E}_{m-1}(-t)\bigr)}{e^{-t}-\mathfrak{E}_{m-1}(-t)}=\sum_{n=0}^\infty\mathfrak{B}_{n,\ge m}\frac{t^n}{n!}\,,
\end{equation*}
respectively, where 
\begin{equation*}
\mathfrak{E}_m(t)=\sum_{n=0}^m\frac{t^n}{n}\,.
\end{equation*}
When $m\to\infty$ in the former case or $m=1$ in the latter case, we have the generating function of the classical Bernoulli numbers $\mathfrak{B}_n$ (with $\mathfrak{B}_1=1/2$), defined by 
\begin{equation*}
\frac{t}{1-e^{-t}}=\sum_{n=0}^\infty\mathfrak{B}_{n}\frac{t^n}{n!}\,.
\end{equation*}
Therefore, we have $\mathfrak{B}_n=\mathfrak{B}_{n,\le\infty}=\mathfrak{B}_{n,\ge 1}$.

The restricted Cauchy numbers $\mathfrak{c}_{n,\le m}$ and the associated Cauchy numbers $\mathfrak{c}_{n,\ge m}$ can be defined by 
\begin{equation*}
\frac{e^{\mathfrak{F}_m(t)}-1}{\mathfrak{F}_m(t)}=\sum_{n=0}^\infty\mathfrak{c}_{n,\le m}\frac{t^n}{n!}
\end{equation*}
and 
\begin{equation*}
\frac{e^{\log(1+t)-\mathfrak{F}_{m-1}(t)}-1}{\log(1+t)-\mathfrak{F}_{m-1}(t)}
=\sum_{n=0}^\infty\mathfrak{c}_{n,\ge m}\frac{t^n}{n!}\,,
\end{equation*}
respectively, where 
\begin{equation*}
\mathfrak{F}_m(t)=\sum_{n=1}^m\frac{(-1)^{n-1}t^n}{n}\,.
\end{equation*}
When $m\to\infty$ in the former case or $m=1$ in the latter case, we have the generating function of the classical Cauchy numbers $\mathfrak{c}_n$, defined by 
\begin{equation*}
\frac{t}{\log(1+t)}=\sum_{n=0}^\infty\mathfrak{c}_{n}\frac{t^n}{n!}\,.
\end{equation*}
Therefore, we have $\mathfrak{c}_n=\mathfrak{c}_{n,\le\infty}=\mathfrak{c}_{n,\ge 1}$. 

A different type of generalization is based upon hypergeometric functions. 
For $N\geq 1$, define hypergeometric Bernoulli numbers $B_{N,n}$ (see \cite{HN1,HN2,Kamano}) by 
\begin{equation*}
\frac{1}{{}_{1}F_{1}(1;N+1;t)}=\frac{t^{N}/N!}{e^{t}-\sum_{n=0}^{N-1}t^{n}/n!}=\sum_{n=0}^{\infty }B_{N,n}\frac{t^{n}}{n!}\,,
\end{equation*}%
where 
\begin{equation*}
{}_{1}F_{1}(a;b;z)=\sum_{n=0}^{\infty }\frac{(a)^{(n)}}{(b)^{(n)}}\frac{z^{n}%
}{n!}
\end{equation*}
is the confluent hypergeometric function with $(x)^{(n)}=x(x+1)\cdots(x+n-1) $ ($n\geq 1$) and $(x)^{(0)}=1$. When $N=1$, $B_{n}=B_{1,n}$ are classical Bernoulli numbers (with $B_{1}=-1/2$) defined by 
\begin{equation*}
\frac{t}{e^{t}-1}=\sum_{n=0}^{\infty }B_{n}\frac{t^{n}}{n!}\,.
\end{equation*}
In addition, define hypergeometric Cauchy numbers $c_{N,n}$ (see \cite{Ko3})
by 
\begin{equation*}
\frac{1}{{}_{2}F_{1}(1,N;N+1;-t)}=\frac{(-1)^{N-1}t^{N}/N}{\log(1+t)-\sum_{n=1}^{N-1}(-1)^{n-1}t^{n}/n}=\sum_{n=0}^{\infty }c_{N,n}\frac{t^{n}}{n!}\,,
\end{equation*}
where 
\begin{equation*}
{}_{2}F_{1}(a,b;c;z)=\sum_{n=0}^{\infty }\frac{(a)^{(n)}(b)^{(n)}}{(c)^{(n)}}\frac{z^{n}}{n!}
\end{equation*}
is the Gauss hypergeometric function. When $N=1$, $\mathfrak{c}_{n}=c_{1,n}$
are classical Cauchy numbers defined by 
\begin{equation*}
\frac{t}{\log (1+t)}=\sum_{n=0}^{\infty }\mathfrak{c}_{n}\frac{t^{n}}{n!}\,.
\end{equation*}
Furthermore, the hypergeometric Bernoulli polynomials \cite{HN1,HN2} are defined by 
\begin{equation}
\frac{\frac{t^{m}}{m!}e^{xt}}{e^{t}-\sum_{j=0}^{m-1}\frac{t^{j}}{j!}}
=\sum_{n=0}^{\infty }B_{m,n}(x)\frac{t^{n}}{n!}  
\label{def:hbp}
\end{equation}%
or 
\begin{equation*}
\frac{e^{xt}}{{}_{1}F_{1}(1;m+1;t)}=\sum_{n=0}^{\infty }B_{m,n}(x)\frac{t^{n}}{n!}\,,
\end{equation*}
and the hypergeometric Cauchy polynomials $c_{M,N,n}(x)$ \cite{Ko3} are defined by 
\begin{equation*}
\frac{1}{(1+t)^{x}}\frac{1}{{}_{2}F_{1}(M,N;N+1;-t)}=\sum_{n=0}^{\infty}c_{M,N,n}(x)\frac{t^{n}}{n!}\,.
\end{equation*}
Observe that in the case $m=0$, formula (\ref{def:hbp}) becomes 
\begin{equation*}
e^{(x-1)t}=\sum_{n=0}^{\infty }B_{0,n}(x)\frac{t^{n}}{n!}\,, 
\end{equation*}
from where we see that $B_{0,n}(x)=(x-1)^{n}$. 

When $m=1$, then $B_{n}(x)=B_{1,n}(x)$ are the classical Bernoulli polynomials, defined by 
\begin{equation*}
\frac{te^{xt}}{e^{t}-1}=\sum_{n=0}^{\infty }B_{n}(x)\frac{t^{n}}{n!}\,,
\end{equation*}
and $c_{n}(x)=c_{1,n}(x)$ are the classical Cauchy polynomials, defined by 
\begin{equation*}
\frac{t}{(1+t)^{x}\log (1+t)}=\sum_{n=0}^{\infty }c_{n}(x)\frac{t^{n}}{n!}\,.
\end{equation*}

In this paper, we give some properties of truncated Euler polynomials.

\section{Some properties of truncated Euler polynomials}

\begin{theorem}  
For $n\ge 1$, we have 
$$
E_{1,n}(x)=2 n(x-1)^{n-1}\,. 
$$ 
\label{th:m=1} 
\end{theorem}
\begin{proof} 
When $m=1$, (\ref{def:tep}) becomes 
\begin{align*} 
\sum_{n=1}^\infty E_{1,n}(x)\frac{t^{n-1}}{n!}&=2 e^{(x-1)t}\\ 
&=2 n\sum_{n=1}^\infty\frac{(x-1)^{n-1}t^{n-1}}{n!}\,. 
\end{align*}  
Comparing the coefficients on both sides, we get the result.  
\end{proof}

\begin{theorem} 
We have 
$$
E_{m,n}(x)=0\quad(n=0,1,\dots,m-1)
$$ 
and 
$$
E_{m,n+m}(x)=2\binom{n+m}{n}x^n-\sum_{j=0}^n\binom{n+m}{j}E_{m,j}(x)\quad(n\ge 0)\,. 
$$ 
\label{th20}
\end{theorem}
\begin{proof} 
From (\ref{def:tep}), we have 
\begin{align*}  
\frac{2 t^m}{m!}\sum_{n=0}^\infty\frac{(x t)^n}{n!}&=\left(\sum_{n=0}^\infty E_{m,n}(x)\frac{t^n}{n!}\right)\left(1+\sum_{j=m}^\infty\frac{t^j}{j!}\right)\\
&=\sum_{n=0}^\infty E_{m,n}(x)\frac{t^n}{n!}+\left(\sum_{n=0}^\infty E_{m,n}(x)\frac{t^n}{n!}\right)\left(\sum_{j=0}^\infty\frac{t^{j+m}}{(j+m)!}\right)\,. 
\end{align*}  
Hence, 
\begin{align*} 
&\left(\sum_{n=0}^\infty E_{m,n}(x)\frac{t^n}{n!}\right)\left(\sum_{j=0}^\infty\frac{t^{j+m}}{(j+m)!}\right)\\ 
&=\sum_{n=0}^\infty\frac{2 x^n t^{n+m}}{n!m!}-\sum_{n=0}^\infty E_{m,n+m}(x)\frac{t^{n+m}}{(n+m)!}-\sum_{n=0}^{m-1}E_{m,n}(x)\frac{t^n}{n!}\\
&=\sum_{n=0}^\infty\left(2\binom{n+m}{m}x^n-E_{m,n+m}\right)\frac{t^{n+m}}{(n+m)!}-\sum_{n=0}^{m-1}E_{m,n}(x)\frac{t^n}{n!}\,. 
\end{align*} 
and 
\begin{align*} 
\left(\sum_{n=0}^\infty E_{m,n}(x)\frac{t^n}{n!}\right)\left(\sum_{j=0}^\infty\frac{t^{j+m}}{(j+m)!}\right)
&=\sum_{n=0}^\infty\sum_{j=0}^n E_{m,j}(x)\frac{t^j}{j!}\frac{t^{n-j+m}}{(n-j+m)!}\\
&=\sum_{n=0}^\infty\left(\sum_{j=0}^n\binom{n+m}{j}E_{m,j}(x)\right)\frac{t^{n+m}}{(n+m)!}\,. 
\end{align*} 
Comparing the coefficients, we get the results. 
\end{proof}

\noindent \textbf{Example.} When $m=2$, we have $E_{2,0}(x)=E_{2,1}(x)=0$.
From the recurrence relation 
\begin{equation*}
E_{2,n+2}(x)=2\binom{n+2}{2}x^{n}-\sum_{j=0}^{n}\binom{n+2}{j}E_{2,j}(x)\,,
\end{equation*}%
by putting $n=0,1,2$, we get 
\begin{align*}
E_{2,2}(x)& =2, \\
E_{2,3}(x)& =6x, \\
E_{2,4}(x)& =12x^{2}-\sum_{j=0}^{2}\binom{4}{j}E_{2,j}(x) \\
& =12x^{2}-\binom{4}{2}\cdot 2=12(x^{2}-1),
\end{align*}

Similarly one gets $E_{2,5}(x)=20(x^{3}-3x-1),$ $%
E_{2,6}(x)=30(x+1)(x^{3}-5x^{2}-x+9),$ and so on.

\begin{theorem}  
$$
E_{m,n}(x+y)=\sum_{j=0}^n\binom{n}{j}E_{m,j}(x)y^{n-j}\,. 
$$ 
\label{th30}  
\end{theorem}  
\begin{proof}  
From (\ref{def:tep}), we get 
\begin{align*}  
\sum_{n=0}^\infty E_{m,n}(x+y)\frac{t^n}{n!}&=\frac{\frac{2 t^m}{m!}e^{(x+y)t}}{e^t+1-\sum_{j=0}^{m-1}\frac{t^j}{j!}}\\
&=\frac{\frac{2 t^m}{m!}e^{x t}}{e^t+1-\sum_{j=0}^{m-1}\frac{t^j}{j!}}e^{y t}\\
&=\left(\sum_{n=0}^\infty E_{m,n}(x)\frac{t^n}{n!}\right)\left(\sum_{n=0}^\infty\frac{y^n t^n}{n!}\right)\\
&=\sum_{n=0}^\infty\sum_{j=0}^n E_{m,j}(x)\frac{t^j}{j!}\frac{y^{n-j}t^{n-j}}{(n-j)!}\\
&=\sum_{n=0}^\infty\left(\sum_{j=0}^n\binom{n}{j}E_{m,j}(x)y^{n-j}\right)\frac{t^n}{n!}\,. 
\end{align*}  
Comparing the coefficients on both sides, we get the desired result. 
\end{proof}

\begin{theorem} 
For $m,n\ge 0$, we have 
$$
E_{m,n}(x)=\sum_{k=0}^n\binom{n}{k}E_{m,n-k}x^k\,. 
$$ 
\label{th501} 
\end{theorem}  
\begin{proof}  
By the definition (\ref{def:tep}), we have 
\begin{align*}  
\sum_{n=0}^\infty E_{m,n}(x)\frac{t^n}{n!}&=\frac{\frac{2 t^m}{m!}}{e^t+1-\sum_{j=0}^{m-1}\frac{t^j}{j!}}e^{x t}\\
&=\left(\sum_{n=0}^\infty E_{m,n}(x)\frac{t^n}{n!}\right)\left(\sum_{n=0}^\infty x^n\frac{t^n}{n!}\right)\\
&=\sum_{n=0}^\infty\sum_{k=0}^n\binom{n}{k}E_{m,n-k}x^k\frac{t^n}{n!}\,.
\end{align*} 
Comparing the coefficients of both sides, we get the desired result.  
\end{proof}  

The generating funcion of the Stirling numbers of the second kind denoted by $\sts{n}{k}$ is given by 
$$
\frac{(e^t-1)^k)}{k!}=\sum_{n=0}^\infty\sts{n}{k}\frac{t^n}{n!}\,.
$$ 
The falling factorial $(x)_n$ and the rising factorial $(x)^{(n)}$ are defined by $(x)_n=x(x-1)\cdots(x-n+1)$ and $(x)^{(n)}=x(x+1)\cdots(x+n-1)$ ($n\ge 1$) with $(x)_0=(x)^{(0)}=1$, respectively.   

\begin{theorem}  
For $m,n\ge 0$, we have 
$$
E_{m,n}(x)=\sum_{\mu=0}^n\sum_{l=\mu}^n\sts{l}{\mu}\binom{n}{l}E_{m,n-l}\cdot(x)_\mu\,. 
$$ 
\label{th:st2-ff} 
\end{theorem}  
\begin{proof}  
Since 
$$
\sum_{\mu=0}^l\sts{l}{\mu}(x)_\mu=x^l\,, 
$$ 
By Theorem \ref{th501}, we have 
\begin{align*} 
E_{m,n}(x)&=\sum_{l=0}^n\binom{n}{l}E_{m,n-l}x^l\\
&=\sum_{l=0}^n\binom{n}{l}E_{m,n-l}\sum_{\mu=0}^l\sts{l}{\mu}(x)_\mu\\
&=\sum_{\mu=0}^n\sum_{l=\mu}^n\sts{l}{\mu}\binom{n}{l}E_{m,n-l}\cdot(x)_\mu
\,. 
\end{align*} 
\end{proof}

\begin{theorem}  
For $m,n\ge 0$, we have 
$$
E_{m,n}(x)=\sum_{\mu=0}^n\sum_{l=\mu}^n\sts{l}{\mu}\binom{n}{l}E_{m,n-l}(-\mu)\cdot(x)^{(\mu)}\,. 
$$ 
\label{th:st2-rf} 
\end{theorem}  
\begin{proof}  
We have 
\begin{align*}  
\sum_{n=0}^\infty E_{m,n}(x)\frac{t^n}{n!}
&=\frac{\frac{2 t^m}{m!}}{e^t+1-\sum_{j=0}^{m-1}\frac{t^j}{j!}}(e^{-t})^{-x}\\
&=\frac{\frac{2 t^m}{m!}}{e^t+1-\sum_{j=0}^{m-1}\frac{t^j}{j!}}\sum_{\mu=0}^\infty\binom{x+\mu-1}{\mu}(1-e^{-t})^{-x}\\
&=\frac{\frac{2 t^m}{m!}}{e^t+1-\sum_{j=0}^{m-1}\frac{t^j}{j!}}\sum_{\mu=0}^\infty(x)^{(\mu)}\frac{(e^t-1)^\mu}{\mu!}e^{-\mu t}\\
&=\sum_{\mu=0}^\infty(x)^{(\mu)}\left(\sum_{n=0}^\infty\sts{n}{\mu}\frac{t^n}{n!}\right)\left(\sum_{n=0}^\infty E_{m,n}(-\mu)\frac{t^n}{n!}\right)\\
&=\sum_{\mu=0}^\infty(x)^{(\mu)}\sum_{n=0}^\infty\binom{n}{l}\sum_{l=0}^n\sts{l}{\mu}E_{m,n-l}(-\mu)\frac{t^n}{n!}\\
&=\sum_{n=0}^\infty\sum_{\mu=0}^n\sum_{l=\mu}^n\sts{l}{\mu}\binom{n}{l}E_{m,n-l}(-\mu)\cdot(x)^{(\mu)}\frac{t^n}{n!}\,. 
\end{align*} 
Comparing the coefficients, we get the desired result.  
\end{proof}

We finish this section by mentioning the relation with Frobenius-Euler polynomials. For $\lambda\in\mathbb C$ with $\lambda\ne 1$ and a nonnegative integer $r$, Frobenius-Euler polynomials $H_n^{(r)}(x|\lambda)$ are defined by 
$$
\left(\frac{1-\lambda}{e^t-\lambda}\right)^r e^{x t}=\sum_{n=0}^\infty H_n^{(r)}(x|\lambda)\frac{t^n}{n!} 
$$ 
(see e.g. \cite{KK}).  


\begin{theorem}  
For $m,n\ge 0$, we have 
$$
E_{m,n}(x)=\sum_{\mu=0}^n\frac{1}{(1-\lambda)^r}\binom{n}{\mu}\sum_{i=0}^r\binom{r}{i}(-1)^{r-i}E_{m,n-\mu}(i)\cdot H_\mu^{(r)}(x|\lambda)\,. 
$$ 
\end{theorem}  
\begin{proof}  
We have 
\begin{align*} 
&\sum_{n=0}^\infty E_{m,n}(x)\frac{t^n}{n!}\\
&=\frac{\frac{2 t^m}{m!}}{e^t+1-\sum_{j=0}^{m-1}\frac{t^j}{j!}}\left(\frac{e^t-\lambda}{1-\lambda}\right)^r\left(\frac{1-\lambda}{e^t-\lambda}\right)^r e^{x t}\\
&=\left(\sum_{n=0}^\infty E_{m,n}\frac{t^n}{n!}\right)\sum_{i=0}^r\binom{r}{i}(e^{t i})(-\lambda)^{r-i}\frac{1}{(1-\lambda)^r}\left(\sum_{n=0}^\infty H_n^{(r)}(x|\lambda)\frac{t^n}{n!}\right)\\
&=\sum_{i=0}^r\binom{r}{i}(-\lambda)^{r-i}\frac{1}{(1-\lambda)^r}(e^{t i})\sum_{n=0}^\infty\sum_{\mu=0}^n\binom{n}{\mu}E_{m,n-\mu}H_\mu^{(r)}(x|\lambda)\frac{t^n}{n!}\\
&=\sum_{n=0}^\infty\sum_{\mu=0}^n\frac{1}{(1-\lambda)^r}\binom{n}{\mu}\sum_{i=0}^r\binom{r}{i}(-1)^{r-i}E_{m,n-\mu}(i)\cdot H_\mu^{(r)}(x|\lambda)\frac{t^n}{n!}\,. 
\end{align*} 
Comparing the coefficients on both sides, we get the desired result.  
\end{proof}

\section{Relations with hypergeometric Bernoulli polynomials}

In this section, we shall show several relations with hypergeometric
Bernoulli polynomials.

\begin{theorem}  
For non-negative integers $n$ and $m$, we have 
\begin{multline*} 
\binom{n+m}{n}\sum_{j=0}^n\binom{n}{j}\left(B_{m,j}(x)y^{n-j}-\frac{1}{2}E_{m,j}(y)x^{n-j}\right)\\
=\frac{1}{2}\sum_{j=0}^{n+m}\binom{n+m}{j}E_{m,j}(y)B_{m,n+m-j}(x)\,. 
\end{multline*} 
\label{th:210} 
\end{theorem} 

\begin{proof}[Proof of Theorem \ref{th:210}]
From (\ref{def:tep}), we have 
\begin{align*}
\frac{z^{m}}{m!}e^{y t}&=\left(\sum_{n=0}^\infty\frac{1}{2}E_{m,n}(y)\frac{t^{n}}{n!}\right)\left(e^t-\sum_{j=0}^{m-1}\frac{t^{j}}{j!}+1\right)\\
&=\left(\sum_{n=0}^\infty\frac{1}{2}E_{m,n}(y)\frac{t^{n}}{n!}\right)\left(e^t-\sum_{j=0}^{m-1}\frac{t^{j}}{j!}\right) 
+\left(\sum_{n=0}^\infty\frac{1}{2}E_{m,n}(y)\frac{t^{n}}{n!}\right)\,.  
\end{align*}
Hence, 
\begin{align*}
&\frac{t^{m}}{m!}e^{y t}\sum_{n=0}^\infty B_{m,n}(x)\frac{t^{n}}{n!}\\ 
&=\left(\sum_{n=0}^\infty\frac{1}{2}E_{m,n}(y)\frac{t^n}{n!}\right)\frac{z^m}{m!}e^{x t}
+\left(\sum_{n=0}^\infty\frac{1}{2}E_{m,n}(y)\frac{t^{n}}{n!}\right)\left(\sum_{n=0}^\infty B_{m,n}(x)\frac{t^{n}}{n!}%
\right)\,,  
\end{align*} 
which can be written as 
\begin{align*}
&\frac{t^{m}}{m!}\left(e^{y t}\sum_{n=0}^\infty B_{m,n}(x)\frac{t^{n}}{n!}-e^{x t}\sum_{n=0}^\infty\frac{1}{2}E_{m,n}(y)\frac{t^{n}}{n!}\right)\\ 
&=\left(\sum_{n=0}^\infty\frac{1}{2}E_{m,n}(y)\frac{t^{n}}{n!}\right)\left(\sum_{n=0}^{\infty }B_{m,n}(x)\frac{t^{n}}{n!}\right)\,.  
\end{align*} 
After the products, we get  
\begin{align*}
&\frac{t^{m}}{m!}\sum_{n=0}^\infty\left(
\sum_{j=0}^{n}\binom{n}{j}\left(B_{m,j}(x)y^{n-j}-\frac{1}{2}E_{m,j}(y)x^{n-j}\right)\right)\frac{t^{n}}{n!}\\
&=\sum_{n=0}^\infty\left(\sum_{j=0}^n\frac{1}{2}\binom{n}{j}E_{m,j}(y)B_{m,n-j}(x)\right)\frac{t^{n}}{n!}\\ 
&=\sum_{n=m}^\infty\left(\sum_{j=0}^n\frac{1}{2}\binom{n}{j}E_{m,j}(y)B_{m,n-j}(x)\right)\frac{t^{n}}{n!}\,.
\end{align*}
Here, we avoided the zero-terms on the right-hand side. 
Thus, 
\begin{align*}
&\frac{1}{m!}\sum_{n=0}^\infty\left(\sum_{j=0}^{n}\binom{n}{j}\left(B_{m,j}(x)y^{n-j}-\frac{1}{2}E_{m,j}(y)x^{n-j}\right)\right)\frac{t^{n+m}}{n!}\\
&=\sum_{n=0}^\infty\left(\sum_{j=0}^{n+m}\frac{1}{2}\binom{n+m}{j}E_{m,j}(y)B_{m,n+m-j}(x)\right)\frac{t^{n+m}}{(n+m)!}\,.
\end{align*}

Comparing the coefficients on both sides, we get 
\begin{align*}
&\frac{1}{m!n!}\sum_{j=0}^{n}\binom{n}{j}\left(B_{m,j}(x)y^{n-j}-\frac{1}{2}E_{m,j}(y)x^{n-j}\right)\\
&=\frac{1}{(n+m)!}\sum_{j=0}^{n+m}\frac{1}{2}\binom{n+m}{j}E_{m,j}(y)B_{m,n+m-j}(x)\,, 
\end{align*} 
from where the desired conclusion follows.  
\end{proof}   

\begin{Cor}  
For a non-negative integer $n$, we have  
\begin{align} 
\sum_{j=0}^{n}\binom{n}{j}E_{j}(y)\left((x-1)^{n-j}+x^{n-j}\right) 
&=2(x+y-1)^{n}
\label{eq:1}\,,\\
\sum_{j=0}^{n}\binom{n}{j}E_{j}(y)\left(B_{n-j}(x-1)+B_{n-j}(x)\right)
&=2B_{n}(x+y-1) 
\label{eq:2}\,,\\
\sum_{j=0}^{n}\binom{n}{j}E_{j}(y)\left(E_{n-j}(x-1)+E_{n-j}(x)\right)
&=2E_{n}(x+y-1) 
\label{eq:3}\,,\\
\sum_{j=0}^n\binom{n}{j}\bigl(B_j(x)y^{n-j}-j(y-1)^{j-1}x^{n-j}\bigr)
&=\sum_{j=0}^n\binom{n}{j}(y-1)^j B_{n-j}(x)
\label{eq:202}\,,\\
\sum_{j=0}^n\binom{n}{j}\bigl(B_j(x)B_{n-j}(y)-j B_{j-1}(y-1)x^{n-j}\bigr)
&=\sum_{j=0}^n\binom{n}{j}B_j(y-1)B_{n-j}(x)\,,
\label{eq:212}\\
\sum_{j=0}^n\binom{n}{j}\bigl(B_j(x)E_{n-j}(y)-j E_{j-1}(y-1)x^{n-j}\bigr)
&=\sum_{j=0}^n\binom{n}{j}E_j(y-1)B_{n-j}(x)
\label{eq:222}\,. 
\end{align}  
\label{cor:205}
\end{Cor} 

We need the following lemma to prove Corollary \ref{cor:205}.

\begin{Lem}  
If the polynomial identity 
$$ 
\sum_{k=0}^{n}a_{n,k}\left(x+\alpha\right)^{k}=\sum_{k=0}^{n}b_{n,k}\left( x+\beta \right) ^{k}
$$ 
holds, then the following identities hold: 
\begin{enumerate} 
\item[(a)]  the Bernoulli polynomial identity 
$$ 
\sum_{k=0}^{n}a_{n,k}B_{k}(x+\alpha)=\sum_{k=0}^{n}b_{n,k}B_{k}(x+\beta)\,, 
$$ 
\item[(b)] the Euler polynomial identity 
$$
\sum_{k=0}^{n}a_{n,k}E_{k}(x+\alpha)=\sum_{k=0}^{n}b_{n,k}E_{k}(x+\beta)\,.
$$ 
\end{enumerate} 
\label{lem110} 
\end{Lem}  
\begin{proof}  
Affirmation (a) is Theorem 1 in \cite{Pi}. Affirmation (b) is obtained from (a) together with a formula expressing Euler polynomials in terms of Bernoulli polynomials. 
\end{proof}  

\begin{proof}[Proof of Corollary \ref{cor:205}]    
Since $B_{0,n}(x)=(x-1)^{n}$ and $E_{0,n}(x)=E_n(x)$, by setting $m=0$ in Theorem \ref{th:210}, we have 
$$ 
\sum_{j=0}^{n}\binom{n}{j}\left((x-1)^{j}y^{n-j}-\frac{1}{2}E_{j}(y)x^{n-j}\right) 
=\frac{1}{2}\sum_{j=0}^{n}\binom{n}{j}E_{j}(y)(x-1)^{n-j}\,, 
$$ 
yielding 
\begin{align*}  
\sum_{j=0}^n\binom{n}{j}E_j(y)\left((x-1)^{n-j}+x^{n-j}\right)
&=2\sum_{j=0}^n\binom{n}{j}(x-1)^j y^{n-j}\\
&=2(x+y-1)^j\,, 
\end{align*} 
which is (\ref{eq:1}). 
The identities (\ref{eq:2}) and (\ref{eq:3}) are obtained from (\ref{eq:1}) by applying Lemma \ref{lem110}. 
On the other hand, since $B_{1,n}(x)=B_n(x)$ and $E_{1,n}(x)=2 n(x-1)^{n-1}$ (Theorem \ref{th:m=1}), by setting $m=1$ in Theorem \ref{th:210}, we have 
\begin{multline*} 
(n+1)\sum_{j=0}^n\binom{n}{j}\left(B_j(x)y^{n-j}-j(y-1)^{j-1}x^{n-j}\right)\\ 
=\sum_{j=0}^n\binom{n+1}{n-j}(j+1)(y-1)^j B_{n-j}(x)\,. 
\end{multline*} 
Dividing $n+1$ on both sides, we get the identity (\ref{eq:202}).  
The identities (\ref{eq:212}) and (\ref{eq:222}) are obtained from (\ref{eq:202}) by applying Lemma \ref{lem110}. 
\end{proof}

\begin{theorem} 
For non-negative integers $n$ and $m$, we have 
\begin{multline*} 
2\sum_{j=0}^n\binom{n}{j}E_{m+1,n-j}(x)y^j-\frac{2 n}{m+1}\sum_{j=0}^{n-1}\binom{n-1}{j}E_{m,n-j-1}(y)x^j\\
=\sum_{j=0}^n\binom{n}{j}E_{m+1,n-j}(x)E_{m,j}(y)\,. 
\end{multline*} 
\end{theorem} 
\label{th:301}

\begin{proof}[Proof of Theorem \ref{th:301}] 
From the definition (\ref{def:tep}),  
\begin{align*}  
&\frac{2 t^{m+1}}{(m+1)!}e^{x t}\\  
&=\left(e^t+1-\sum_{j=0}^{m-1}\frac{t^j}{j!}-\frac{t^m}{m!}\right)\sum_{n=0}^\infty E_{m+1,n}(x)\frac{t^n}{n!}\\
&=\left(e^t+1-\sum_{j=0}^{m-1}\frac{t^j}{j!}\right)\sum_{n=0}^\infty E_{m+1,n}(x)\frac{t^n}{n!}-\frac{t^m}{m!}\sum_{n=0}^\infty E_{m+1,n}(x)\frac{t^n}{n!}\,.
\end{align*}  
By using the definition (\ref{def:tep}) again,  we have 
\begin{multline*} 
\frac{2 t^{m+1}}{(m+1)!}e^{x t}\sum_{n=0}^\infty E_{m,n}(y)\frac{t^n}{n!}\\
=\frac{2 t^m}{m!}e^{y t}\sum_{n=0}^\infty E_{m+1,n}(x)\frac{t^n}{n!}
-\frac{t^m}{m!}\sum_{n=0}^\infty E_{m+1,n}(x)\frac{t^n}{n!}\sum_{n=0}^\infty E_{m,n}(y)\frac{t^n}{n!}
\end{multline*} 
or 
\begin{align*}  
&\frac{2 n}{m+1}\sum_{n=1}^\infty\left(\sum_{j=0}^{n-1}\binom{n-1}{j}E_{m,n-j-1}(y)x^j\right)\frac{t^n}{n!}\\
&=2\sum_{n=1}^\infty\left(\sum_{j=0}^n\binom{n}{j}E_{m+1,n-j}(x)y^j\right)\frac{t^n}{n!}\\
&\qquad -\sum_{n=1}^\infty\left(\sum_{j=0}^n\binom{n}{j}E_{m+1,n-j}(x)E_{m,j}(y)\right)\frac{t^n}{n!}\,. 
\end{align*}  
Comparing the coefficients on both sides, we obtain the desired result. 
\end{proof} 

\begin{Cor}  
For a non-negative integer $n$, we have  
\begin{align} 
&\sum_{j=0}^{n}\binom{n}{j}(2(x-1)^{n-j}y^j-E_{n-j}(y)x^j)=\sum_{j=0}^{n}\binom{n}{j}(x-1)^{n-j}E_j(y)
\label{eq:302}\,,\\
&\sum_{j=0}^{n}\binom{n}{j}(2B_{n-j}(x-1)y^j-E_{n-j}(y)B_j(x))=\sum_{j=0}^{n}\binom{n}{j}B_{n-j}(x-1)E_j(y)
\label{eq:312}\,,\\
&\sum_{j=0}^{n}\binom{n}{j}(2E_{n-j}(x-1)y^j-E_{n-j}(y)E_j(x))=\sum_{j=0}^{n}\binom{n}{j}E_{n-j}(x-1)E_j(y)
\label{eq:322}\,. 
\end{align}  
\label{cor:305}
\end{Cor}  
\begin{proof}   
Since $E_{1,n}(x)=2 n(x-1)^{n-1}$ (Theorem \ref{th:m=1}), by setting $m=0$ in Theorem \ref{th:301}, we get (\ref{eq:302}).  
The identities (\ref{eq:312}) and (\ref{eq:322}) are obtained from (\ref{eq:302}) by applying Lemma \ref{lem110}.  
\end{proof}



\end{document}